\documentclass[11pt]{amsart}
\usepackage{amsmath, amsthm, mathabx,amscd, amsfonts, amssymb, hyperref, mathrsfs, textcomp, epsfig, graphicx, a4wide, IEEEtrantools, tikz, verbatim, xypic }

\usepackage[all,cmtip]{xy}
\usetikzlibrary{matrix, decorations.pathreplacing,angles,quotes,cd}

\newtheorem{thm}{Theorem}
\newtheorem*{question}{Question}

\newtheorem{lemma}[thm]{Lemma}

\theoremstyle{definition}

\usepackage{mathtools}

     \RequirePackage{rotating}                   
    \def\HMt{%
       \setbox0=\hbox{$\widehat{\mathit{HM}}$}
       \setbox1=\hbox{$\mathit{HM}$}
       \dimen0=1.1\ht0
       \advance\dimen0 by 1.17\ht1
       \smash{\mskip2mu\raise\dimen0\rlap{%
          \begin{turn}{180}
              {$\widehat{\phantom{\mathit{HM}}}$}
           \end{turn}} \mskip-2mu    
                \mathit{HM}
    }{\vphantom{\widehat{\mathit{HM}}}}{}}

\theoremstyle{remark}
\newtheorem*{remark}{Remark}

\begin{document}
\title{Divergence-free framings of three-manifolds via eigenspinors}

\author{Francesco Lin}
\address{Department of Mathematics, Columbia University} 
\email{flin@math.columbia.edu}

\maketitle

 {\centering\textit{Dedicated to Paolo Lisca in the occasion of his 60th birthday.}\par}

\begin{abstract}Gromov used convex integration to prove that any closed orientable three-manifold equipped with a volume form admits three divergence-free vector fields which are linearly independent at every point. We provide an alternative proof of this (inspired by Seiberg-Witten theory) using geometric properties of eigenspinors in three dimensions. In fact, our proof shows that for any Riemannian metric, one can find three divergence-free vector fields such that at every point they are orthogonal and have the same non-zero length.
\end{abstract}

\vspace{0.5cm}

The following classical result of Stiefel is fundamental in three-manifold topology.
\begin{thm}[\cite{Sti}]\label{parallelizable} Every closed orientable three-manifold $Y$ admits a framing, i.e. three vector fields $X_1,X_2$ and $X_3$ which are linearly independent at every point.
\end{thm}

The hardest part of the standard proofs of such result is to establish that the second Stiefel-Whitney class $w_2(TM)$ vanishes; after this, it follows from obstruction theory because $\pi_2(\mathrm{SO}(3))=0$ (see \cite[Ch. 12]{MS}). For alternative `bare hands' proofs, see \cite{BL}.
\par
It is natural to ask whether, in the presence of an additional geometric structure on $Y$, the framing can be chosen to be compatible with it. In this direction, we have the following result of Gromov.
\begin{thm}[\cite{Gro}, p. 182]\label{gromov} Every closed orientable three-manifold $Y$ equipped with a volume form $\Omega$ admits a framing $X_1,X_2$ and $X_3$ consisting of divergence-free vector fields.
\end{thm}
Recall that the divergence $\mathrm{div}(X)$ of a vector field $X$ (with respect to the volume form $\Omega$) is defined in terms of the Lie derivative by
\begin{equation*}
\mathcal{L}_X\Omega=\mathrm{div}(X)\cdot\Omega;
\end{equation*}
a vector field is \textit{divergence-free} if its divergence vanishes, or equivalently if its associated flow is volume-preserving.
\par
If one fixes a Riemannian metric $g$ on $Y$ (and considers the volume form $d\mathrm{vol}_g$), the following is a very natural question with implications in hyperk\"ahler geometry due to Bryant (see also \cite{FLS}).
\begin{question}[\cite{Bry}, Remark 3]
Which closed orientable Riemannian three manifolds $(Y,g)$ admit a divergence-free framing $X_1,X_2$ and $X_3$ which is orthonormal at every point?
\end{question}

Our main goal is to show that if one relaxes the condition of orthonormality to orthonormality up to scaling then such a framing can always be found.

\begin{thm}\label{main}Every closed orientable three-manifold $Y$ equipped with a Riemannian metric $g$ admits a framing $X_1,X_2$ and $X_3$ consisting of divergence-free vector fields so that at every point $p$ in $Y$, $X_1(p),X_2(p)$ and $X_3(p)$ are orthogonal and have the same length.
\end{thm} 

This recovers Gromov's result because any volume form $\Omega$ is the volume form of some Riemannian metric. While Gromov's proof is based on h-principles and in particular convex integration techniques (see also \cite[Chapter 20]{EM} for an exposition), our approach is inspired by Seiberg-Witten theory and is based on elliptic PDEs, in the sense that it uses geometric properties of eigenspinors in dimension three. It is not clear whether the convex integration approach can be adapted to prove Theorem \ref{main}; notice that the geometric setup of our result is much more rigid because it involves three differential equations in four (rather than nine) variables.

\subsection*{Preliminaries on spin Dirac operators. }We begin by recalling some basic facts in spin geometry; we refer the reader to \cite{Roe} for a general discussion and \cite{KM} for a treatment specific for our three-dimensional needs. We will begin by choosing a spin structure on $Y$, which exists because $TY$ is trivial (Theorem \ref{parallelizable} above). Now $\mathrm{Spin(3)}=\mathrm{SU}(2)$, and the spinor representation is given by the natural vector representation on $\mathbb{C}^2$. We denote the associated (rank $2$ hermitian) spinor bundle by $S\rightarrow Y$; this is equipped with the spin connection $\nabla$. The associated Clifford multiplication provides an identification
\begin{equation*}
\rho: TY\rightarrow \mathfrak{su}(S)
\end{equation*}
such that for each oriented orthonormal frame $e_1,e_2,e_3$ at a point $p$, we can find a basis of $S_p$ such that $\rho(e_i)=\sigma_i$ where
\begin{equation}\label{pauli}
\sigma_1=\begin{bmatrix}
i&0\\
0&-i
\end{bmatrix},\quad
\sigma_2=\begin{bmatrix}
0&-1\\
1&0
\end{bmatrix},\quad
\sigma_3=\begin{bmatrix}
0&i\\
i&0
\end{bmatrix},
\end{equation}
are the Pauli matrices. The spin Dirac operator
\begin{equation*}
D:\Gamma(S)\rightarrow\Gamma(S)
\end{equation*}
is given by the composition
\begin{equation*}
\Gamma(S)\stackrel{\nabla}{\longrightarrow}\Gamma(T^*Y\otimes S)\stackrel{\rho}{\rightarrow}\Gamma(S)
\end{equation*}
where we extended $\rho$ to $1$-forms via the musical isomorphism. The spin Dirac operator is a first-order elliptic formally self-adjoint operator, and therefore (given that $Y$ is closed) diagonalizable in $L^2$ with real discrete spectrum infinite in both directions. We will be particularly interested in its eigenspinors, i.e. non-zero solutions to the eigenvalue equation
\begin{equation}\label{eigen}
D\Psi=\lambda\Psi,
\end{equation}
especially in the situation of $\lambda\neq 0$.

\subsection*{The quadratic map.} Inspired by the three-dimensional Seiberg-Witten equations, given any section $\Psi\in\Gamma(S)$ we can consider the traceless hermitian endomorphism $(\Psi\Psi^*)_0\in\Gamma(i\mathfrak{su}(S))$. In coordinates, if $\Psi=(\alpha,\beta)$, then
\begin{equation*}
(\Psi\Psi^*)_0=\begin{bmatrix}
\frac{1}{2}(|\alpha|^2-|\beta|^2)& \alpha\bar{\beta}\\
\bar{\alpha}\beta&\frac{1}{2}(|\beta|^2-|\alpha|^2)
\end{bmatrix}.
\end{equation*}
The key computation (also inspired by the Seiberg-Witten equations) for our purposes is the following.
\begin{lemma}\label{closed}
If $\Psi$ is an eigenspinor, then the vector field $X:=\rho^{-1}(i(\Psi\Psi^*)_0)$ divergence-free.
\end{lemma}
\begin{proof}
We will check that the statement holds at any fixed $p$ in $Y$. Fix a local orthonormal frame $e_i$, which we assume to be syncronous at $p$, i.e. $\nabla_{e_i}e_j(p)=0$. Using \ref{pauli}) see that
\begin{equation}\label{quadratic}
X=\rho^{-1}(i(\Psi\Psi^*)_0)=\frac{1}{2}(|\alpha^2|-|\beta|^2)e_1+\mathrm{Im}({\alpha}\bar{\beta})e_2+\mathrm{Re}({\alpha}\bar{\beta})e_3.
\end{equation}
In these coordinates, the eigenvalue equation (\ref{eigen}) is
\begin{align*}
i\alpha_1-\beta_2+i\beta_3&=\lambda\alpha\\
\alpha_2+i\alpha_3-i\beta_1&=\lambda\beta.
\end{align*}
Using that the frame is syncronous at $p$, we compute the divergence of $X$ at $p$ as follows:
\begin{align*}
\mathrm{Re}(\bar{\alpha}\alpha_1)-\mathrm{Re}(\bar{\beta}\beta_1)+\mathrm{Im}(\alpha_2\bar{\beta})+\mathrm{Im}(\alpha\bar{\beta}_2)+\mathrm{Re}(\alpha_3\bar{\beta})+\mathrm{Re}(\alpha\bar{\beta}_3)&=\\
\mathrm{Re}(\alpha_1\bar{\alpha})-\mathrm{Re}(\beta_1\bar{\beta})-\mathrm{Re}(i\alpha_2\bar{\beta})+\mathrm{Re}(i{\beta}_2\bar{\alpha})+\mathrm{Re}(\alpha_3\bar{\beta})+\mathrm{Re}({\beta}_3\bar{\alpha})&=\\
\mathrm{Re}\left((-\beta_1-i\alpha_2+\alpha_3)\bar{\beta}\right)+\mathrm{Re}\left((\alpha_1+i\beta_2+\beta_3)\bar{\alpha}\right)&=\\
\mathrm{Re}\left((-i\lambda\beta)\bar{\beta}\right)+\mathrm{Re}\left((-i\lambda\alpha)\bar{\alpha}\right)&=0
\end{align*}
where we used $\mathrm{Im}(z)=-\mathrm{Re}(iz)$ and that $\lambda$ is real.\end{proof}
\begin{remark}
The result is still true if we consider more generally spin$^c$ Dirac operators $D_B$ (as it is customary in Seiberg-Witten theory). Indeed, we performed the computation pointwise, and any spin$^c$ connection $B$ can be made into the spin connection at a point via a gauge transformation. Furthermore, we can also allow $\lambda$ to be any real valued function on $Y$.\end{remark}

\subsection*{The quaternionic structure.} A fundamental feature of the spin Dirac operator in three-dimensions is its additional quaternionic structure (see for example \cite[Ch. 5]{Lin} for more details). Namely, we can identify the spinor representation as
\begin{align*}
\mathbb{C}^2&\equiv\mathbb{H}\\
(v,w)&\mapsto v+jw
\end{align*}
and consider the right action of $\mathbb{H}$ by multiplication; in particular, complex scalars act as usual while the action of $j$ under identification is given by
\begin{equation*}
(v,w)\cdot j=(-\bar{w},\bar{v}).
\end{equation*}
This induces a complex antilinear map squaring to $-1$ on the spinor bundle $S\rightarrow Y$ (i.e. a quaternionic structure) which we still denote by $j$. The spin Dirac operator $D$ is compatible with this action in the sense that
\begin{equation*}
D(\Psi\cdot j)=(D\Psi)\cdot j.
\end{equation*}
In particular, its eigenspaces are naturally equipped with a quaternionic structure (hence are even dimensional as complex vector spaces). In what follows, we will say that an eigenvalue $D$ is \textit{simple} if the corresponding eigenspace is one dimensional over $\mathbb{H}$.

\subsection*{Geometry of eigenspinors.} With this in mind, we will now state the two main results \cite{Dah}, \cite{Her} about the geometry of eigenspinors on three-manifolds that will be fundamental for our purposes: informally speaking, for a generic metric the spin Dirac operator has no kernel and only simple eigenvalues; furthermore all eigenspinors are nowhere vanishing. Intuitively speaking, the latter should be expected as the spinor bundle $S\rightarrow Y$ has real rank $4$. Of course, the proof of such results is quite technical in nature as the Dirac operator depends on the metric in a complicated way. Furthermore, we will need the following more refined version for our purposes.

\begin{thm}[\cite{Dah}, \cite{Her}]\label{generic}
Consider a closed three-manifold $Y$ equipped with a Riemannian metric $g$ and a spin structure. Then for a generic metric $g'$ conformal to $g$, all non-zero eigenvalues of the spin Dirac operator are simple, and all eigenspinors corresponding to non-zero eigenvalues are nowhere vanishing.
\end{thm}
It is important in the statement to focus on \textit{non-zero} eigenvalues, because the kernel of $D$ (i.e. the space of harmonic spinors) is conformally invariant \cite{Hit}. On the other hand, for a generic metric (not necessarily conformal to a given one) the kernel is trivial \cite{Mai}. Notice that while the main statements of \cite{Dah} and \cite{Her} concern the space of all metrics, the proof is based on a careful analysis of a given conformal class; in particular the result we stated consists of Remark 1.3 in \cite{Dah} and Theorem 4.3 in \cite{Her}. Finally, for our purposes we will only need the statement that non-harmonic eigenspinors have no zeroes, but we emphasized the role of simple eigenvalues as it is an assumption in its proof.

\subsection*{Proof of the main result}
Fix a spin structure and choose a metric $g'$ conformal to $g$ such that the conclusion of Theorem \ref{generic} holds, and consider an eigenspinor $\Psi'$ corresponding to an eigenvalue $\lambda'\neq0$. Using the quaternionic structure, we consider then the three $\lambda'$-eigenspinors
\begin{equation*}
\Psi'_1=\Psi',\quad \Psi_2'=\Psi'\cdot \frac{1+k}{\sqrt{2}},\quad \Psi'_3=\Psi'\cdot \frac{1+j}{\sqrt{2}}
\end{equation*}
all of which are nowhere vanishing (here $k=ij\in\mathbb{H})$. By Lemma \ref{closed} the quadratic map associates to them nowhere-vanishing vector fields $X'_1,X'_2,X'_3$ which are divergence-free (with respect to $d\mathrm{vol}_{g'}$). Furthermore, they are readily checked to be orthogonal and to have the same length with respect to $g'$ at every point. Indeed, we can identify $S_p'\equiv\mathbb{C}^2\equiv \mathbb{H}$ by setting
\begin{equation*}
\Psi'\equiv(a,0)\text{ and } \Psi'\cdot j\equiv(0,a) \text{ where }a=|\Psi'(p)|\in\mathbb{R}^{>0}.
\end{equation*}
This determines a $g'$-orthonormal basis of $T_pY$ (denoted by $\{e'_i\}$) via the identification (\ref{pauli}). Then we have that at the point we can identify the three spinors as
\begin{equation*}
\Psi'_1=(a,0),\quad \Psi'_2=(\frac{a}{\sqrt{2}},\frac{-ia}{\sqrt{2}}),\quad \Psi'_3=(\frac{a}{\sqrt{2}},\frac{a}{\sqrt{2}}),
\end{equation*}
which correspond via the quadratic map to the vectors
\begin{equation*}
X_1'=\frac{a^2}{2}e'_1,\quad X_2'=\frac{a^2}{2}e'_2, \quad X_3'=\frac{a^2}{2}e'_3
\end{equation*}
respectively. Finally, we can write $g=f^2 g'$ for some positive function $f$, and the vector fields
\begin{equation*}
X_i:=\frac{1}{f^3}X'_i
\end{equation*}
are divergence-free with respect to $d\mathrm{vol}_g=f^3\cdot d\mathrm{vol}_{g'}$ because by Cartan's formula
\begin{equation*}
\mathcal{L}_X\Omega=d(\iota_X\Omega)+\iota_Xd\Omega=d(\iota_X\Omega)
\end{equation*}
the vector field $X$ is divergence-free with respect to $\Omega$ if and only if the $2$-form $\iota_X\Omega$ is closed.

\begin{remark}Notice that the proofs of Theorems \ref{gromov} and \ref{main} both take as input Theorem \ref{parallelizable}. Indeed, a key ingredient in our proof is the existence of a spin structure, which is equivalent to $w_2(TY)=0$. On the other hand, Gromov's approach shows that any framing of $Y$ is homotopic (through framings) to a framing by divergence-free vector fields. It is an interesting question to understand which homotopy classes of framings admit representatives as in Theorem \ref{main}. Referring to \cite{KirM} for details, given a framing all other ones are classified up to homotopy by the set of homotopy classes $[Y,\mathrm{SO}(3)]$. To a homotopy class one can associate an element
\begin{equation*}
\mathrm{Hom}(\pi_1(Y),\pi_1(\mathrm{SO}(3)))=H^1(Y;\mathbb{Z}/2)
\end{equation*}
which corresponds to the underlying spin structure. Our proof shows that any spin structure admits a framing as in Theorem \ref{main}. On the other hand, the homotopy classes inducing the same spin structure form an affine space over
\begin{equation*}
H^3(Y;\pi_3(\mathrm{SO}(3)))=\mathbb{Z},
\end{equation*}
and it is not clear from our approach whether all of them can be realized. More in general, it is an interesting question to understand the topological features of eigenspinors on three-manifolds for generic metrics.
\end{remark}

\vspace{0.5cm}

\textit{Acknowledgements.} The author thanks Robert Bryant and Yakov Eliashberg for some useful comments. This work was partially supported by NSF grant DMS-2203498.

\bibliographystyle{alpha}
\bibliography{biblio}

\begin{thebibliography}{CEM24}

\bibitem[BL18]{BL}
Riccardo Benedetti and Paolo Lisca.
\newblock Framing 3-manifolds with bare hands.
\newblock {\em Enseign. Math.}, 64(3-4):395--413, 2018.

\bibitem[Bry10]{Bry}
Robert~L. Bryant.
\newblock Non-embedding and non-extension results in special holonomy.
\newblock In {\em The many facets of geometry}, pages 346--367. Oxford Univ.
  Press, Oxford, 2010.

\bibitem[CEM24]{EM}
K.~Cieliebak, Y.~Eliashberg, and N.~Mishachev.
\newblock {\em Introduction to the {$h$}-principle}, volume 239 of {\em
  Graduate Studies in Mathematics}.
\newblock American Mathematical Society, Providence, RI, second edition, [2024]
  \copyright 2024.

\bibitem[Dah03]{Dah}
Mattias Dahl.
\newblock Dirac eigenvalues for generic metrics on three-manifolds.
\newblock {\em Ann. Global Anal. Geom.}, 24(1):95--100, 2003.

\bibitem[FLS17]{FLS}
Joel Fine, Jason~D. Lotay, and Michael Singer.
\newblock The space of hyperk\"{a}hler metrics on a 4-manifold with boundary.
\newblock {\em Forum Math. Sigma}, 5:Paper No. e6, 50, 2017.

\bibitem[Gro86]{Gro}
Mikhael Gromov.
\newblock {\em Partial differential relations}, volume~9 of {\em Ergebnisse der
  Mathematik und ihrer Grenzgebiete (3) [Results in Mathematics and Related
  Areas (3)]}.
\newblock Springer-Verlag, Berlin, 1986.

\bibitem[Her14]{Her}
Andreas Hermann.
\newblock Zero sets of eigenspinors for generic metrics.
\newblock {\em Comm. Anal. Geom.}, 22(2):177--218, 2014.

\bibitem[Hit74]{Hit}
Nigel Hitchin.
\newblock Harmonic spinors.
\newblock {\em Advances in Math.}, 14:1--55, 1974.

\bibitem[KM99]{KirM}
Rob Kirby and Paul Melvin.
\newblock Canonical framings for {$3$}-manifolds.
\newblock In {\em Proceedings of 6th {G}\"{o}kova {G}eometry-{T}opology
  {C}onference}, volume~23, pages 89--115, 1999.

\bibitem[KM07]{KM}
Peter Kronheimer and Tomasz Mrowka.
\newblock {\em Monopoles and three-manifolds}, volume~10 of {\em New
  Mathematical Monographs}.
\newblock Cambridge University Press, Cambridge, 2007.

\bibitem[Lin18]{Lin}
Francesco Lin.
\newblock A {M}orse-{B}ott approach to monopole {F}loer homology and the
  triangulation conjecture.
\newblock {\em Mem. Amer. Math. Soc.}, 255(1221):v+162, 2018.

\bibitem[Mai97]{Mai}
Stephan Maier.
\newblock Generic metrics and connections on {S}pin- and
  {S}pin{$^c$}-manifolds.
\newblock {\em Comm. Math. Phys.}, 188(2):407--437, 1997.

\bibitem[MS74]{MS}
John~W. Milnor and James~D. Stasheff.
\newblock {\em Characteristic classes}, volume No. 76 of {\em Annals of
  Mathematics Studies}.
\newblock Princeton University Press, Princeton, NJ; University of Tokyo Press,
  Tokyo, 1974.

\bibitem[Roe98]{Roe}
John Roe.
\newblock {\em Elliptic operators, topology and asymptotic methods}, volume 395
  of {\em Pitman Research Notes in Mathematics Series}.
\newblock Longman, Harlow, second edition, 1998.

\bibitem[Sti35]{Sti}
E.~Stiefel.
\newblock Richtungsfelder und {F}ernparallelismus in n-dimensionalen
  {M}annigfaltigkeiten.
\newblock {\em Comment. Math. Helv.}, 8(1):305--353, 1935.

\end{thebibliography}

\end{document}